\documentclass[a4paper, draft, 12pt]{amsproc}
\usepackage{amssymb}
\usepackage{amsmath,amssymb,ulem,color}

\usepackage[hyphens]{url} 
\usepackage[dvips]{graphicx} 
\usepackage{cmdtrack} 
\usepackage{mathrsfs}                       

\theoremstyle{plain}
 \newtheorem{theorem}{Theorem}[section]

 \newtheorem{cor}{Corollary}[section]
\theoremstyle{Definition}
 \newtheorem{exm}{Example}[section]
 
 \newtheorem{dfn}{Definition}[section]
\theoremstyle{remark}

 \numberwithin{equation}{section}

\renewcommand{\leq}{\leqslant}
\renewcommand{\geq}{\geqslant}

\title[On  Completely Invariant Julia Sets of Trascendental Semigroups]{On Completely Invariant Julia Sets of Transcendental Semigroups}



\author[B. H. Subedi]{\bfseries  Bishnu Hari Subedi}

\address{ %% Put here your affiliation; street address is not required
Central Department of Mathematics \\ % \hfill (Received 00 00 201?)\\
Institute of Science and Technology   \\ %\hfill (Revised  00 00 201?)\\
Tribhuvan University   \\ %\hfill (Revised  00 00 201?)\\
Kirtipur, Kathmandu\\
Nepal}
\email{subedi.abs@gmail.com / subedi\_bh@cdmathtu.edu.np }

\author[A. Singh]{Ajaya Singh}
\address{Central Department of Mathematics, Institute of Science and Technology, Tribhuvan University, Kirtipur, Kathmandu, Nepal }
\email{singh.ajaya1@gmail.com / singh\_a@cdmathtu.edu.np} 
\thanks{This research work of first author is supported by PhD faculty fellowship from University Grants Commission, Nepal.} %% optional
\begin{document}

{\begin{flushleft}\baselineskip9pt\scriptsize
%PUBLICATIONS DE L'INSTITUT MATH\'EMATIQUE\newline
%Nouvelle s\'erie, tome ??(1??)) (201?), od--do \hfill DOI: \\
%MANUSCRIPT
\end{flushleft}}
\vspace{18mm} \setcounter{page}{1} \thispagestyle{empty}

\begin{abstract}
In holomorphic semigroup dynamics, Julia set is in general backward invariant and so some fundamental results of classical complex dynamics can not be generalized to semigroup dynamics. In this paper, we define completely invariant Julia set of transcendental semigroup and we see how far the results of classical transcendental dynamics generalized to transcendental semigroup dynamics. 
\end{abstract}

\maketitle

\section{Introduction}

Let us denote the class of all rational maps on $ \mathbb{C_{\infty}} $ by $ \mathscr{R} $ and class of all transcendental entire maps on $ \mathbb{C} $ by $ \mathscr{E} $. 
Our particular interest is to study of dynamics of the families of above two classes of complex analytic (holomorphic) maps.   For a collection $\mathscr{F} = \{f_{\alpha}\}_{\alpha \in \Delta} $ of such maps, let 
$$
S =\langle f_{\alpha} \rangle
$$ 
be a \textit{holomorphic semigroup} generated by them. Here $ \mathscr{F} $ is either a collection $ \mathscr{R} $ of rational maps or a collection $ \mathscr{E} $ of transcendental entire maps. $ \Delta $ is an index set to which $ \alpha $  belongs is  finite or infinite. 
Each $f \in S$ is constructed through the composition of finite number of functions $f_{\alpha_k},\;  (k=1, 2, 3,\ldots, m) $. That is, $f =f_{\alpha_1}\circ f_{\alpha_2}\circ f_{\alpha_3}\circ \cdots\circ f_{\alpha_m}$. In particular, if $ f_{\alpha} \in \mathscr{R} $, we say $ S =\langle f_{\alpha} \rangle$ a rational semigroup and if  $ f_{\alpha} \in \mathscr{E} $, we say $ S =\langle f_{\alpha} \rangle$ a transcendental semigroup. 

A semigroup generated by finitely many holomorphic functions $f_{i}, (i = 1, 2, \ldots, \\ n) $  is called \textit{finitely generated  holomorphic semigroup}. We write $S= \langle f_{1},f_{2},\ldots,f_{n} \rangle$.
 If $S$ is generated by only one holomorphic function $f$, then $S$ is \textit{cyclic semigroup}. We write $S = \langle f\rangle$. In this case, each $g \in S$ can be written as $g = f^n$, where $f^n$ is the nth iterates of $f$ with itself. Note that in our study of  semigroup dynamics, we say $S = \langle f\rangle$  a \textit{trivial semigroup}. 
  
In classical complex dynamics, each of Fatou set and Julia set are defined in two different but equivalent ways.  In first definition, Fatou set is defined as the set of normality of the iterates of given function and Julia set is defined as the complement of the Fatou set. The second definition of  Fatou set is given as a largest completely invariant open set and Julia set is given as a smallest completely invariant close set.  
Each of these definitions can be naturally extended to the settings of holomorphic semigroup $ S $. In this paper, we see that these extension definitions are not equivalent in holomorphic semigroup dynamics.  
Based on above first definition (that is,  on the Fatou-Julia theory of a complex analytic function), the Fatou set and Julia set are defined as follows. 
 \begin{dfn}[\textbf{Fatou set, Julia set}]\label{2ab} 
\textit{Fatou set} of the semigroup $S$ is defined by
  \[F (S) = \{z \in \mathbb{C}: S\;\ \textrm{is normal in a neighborhood of}\;\ z\}\] 
and the \textit{Julia set} $J(S) $ of $S$ is its compliment.  Any maximally connected subset $ U $ of the Fatou set $ F(S) $ is called \textit{Fatou component}.     
\end{dfn} 
If $S = \langle f\rangle$, then $F(S), J(S)$ and $I(S)$ are respectively the Fatou set and Julia set in classical complex dynamics. In this situation we simply write: $F(f)$ and $J(f)$. 
From the definition \ref{2ab}, it is clear that  $F(S)$ is the open set and therefore, it complement $ J(f) $ is closed set. Indeed, these definitions generalize the definitions of Julia set and Fatou set of the iteration of single holomorphic map. 

The fundamental contrast between classical complex dynamics and semigroup dynamics appears by different algebraic structure of corresponding semigroups. In fact, non-trivial semigroup (rational or transcendental) need not be, and most often will not be abelian. However, trivial semigroup is cyclic and therefore abelian. As we discussed before, classical complex dynamics is a dynamical study of trivial (cyclic)  semigroup whereas semigroup dynamics is a dynamical study of non-trivial semigroup. 
 
Note that for any holomorphic semigroup $ S $, we have
\begin{enumerate}  
 \item $F(S) \subset F(f)$ for all $f \in S$  and hence  $F(S)\subset \bigcap_{f\in S}F(f)$. 
\item $ J(f) \subset J(S) $ for all $f \in S$.
 \end{enumerate}

\begin{dfn}[\textbf{Forward, backward and completely invariant set}]
Let $ f $ is a map of a set $ X $ into itself.  A subset $U\subset X$ is said to be 
\begin{enumerate}
\item forward invariant\index{forward ! invariant set} under $ f $ if $f(U)\subset U$;
\item backward invariant\index{backward ! invariant set} under $ f $ if $f^{-1}(U) = \{z \in \mathbb{C}: f(z) \in U \}\subset U$; 
\item completely invariant\index{completely invariant ! set} under $ f $ if it is both forward and backward invariant.
\end{enumerate}
\end{dfn}
It is well known that Fatou set $ F(f) $ and Julia Set $ J(f) $ of any holomorphic map $ f $ are completely invariant. If $ S $ is a rational semigroup, then Hinkkanen and Martin {\cite [Theorem 2.1]{hin}} proved the following result.
\begin{theorem}\label{fi}
The Fatou set $ F(S) $ is forward invariant under each element of $ S $ and Julia set $ J(S) $ is backward invariant for each element of $ S $.
\end{theorem}

If $ S $ is a transcendental semigroup, then K. K. Poon  {\cite[Theorem 2.1] {poo}} proved the following result.
\begin{theorem}\label{fi1}
The Fatou set $ F(S) $ is forward invariant for each element of $ S $ and Julia set $ J(S) $  is backward invariant for each element of $ S $.
\end{theorem}

Indeed,  theorems \ref{fi} and \ref{fi1} are contrasting to classical complex dynamics. We generalize the classical completely  invariant notion of Fatou set and Julia set of single function to the completely  invariant  notion of these sets in semigroup dynamics. That is, we make completely invariant Fatou set and Julia set  of semigroup $ S $  in different way under each element of $ S $.  In rational semigroup dynamics, Rich Stankewitz \cite{stan} and Rich Stankewitz, Toshiyori Sugawa and Hiroki Sumi \cite{stan1} studied completely invariant Julia set and Fatou set. They ({\cite[Definition 3]{stan}} and  {\cite[Definitions 1.3 and 1.4]{stan1}}) defined it as follows. 

Let $ S $ be a rational semigroup. The completely invariant Julia set of $ S $ is the set of the form 
$$
E(S) = \bigcap\{G: G\; \text{is closed, completely invariant under each} \; g \in S, \#(G) \geq 3\}
$$ 
where $ \#(G) $ denote the cardinality of $ G $. The completely invariant Fatou set $ W(S)  $  of $ S $ to be the complement of $ E(S) $. That is, $ W(S) =\mathbb{C_{\infty}}- E(S) $.

By the similar fashion, we can think completely invariant Julia set and Fatou set of transcendental semigoup.
\begin{dfn}\label{cir}
Let $ S $ be a transcendental semigroup. We define completely invariant Julia set by
$$
E(S) = \bigcap\{E: E\; \text{is closed, completely invariant set under each}\; f \in S\}
$$
The completely invariant Fatou set $ W(S) $ is defined as the complement of $ E(S) $ in $ \mathbb{C} $. That is, $ W(S) =\mathbb{C}- E(S) $.
\end{dfn}

Note that in transcendental semigroup $ S $,  $E(S)$ exists, is closed and completely invariant under each $ f \in S $ and it contains the Julia set of each element of $ S $. The corresponding Fatou set is open, completely invariant and contained in the Fatou set of each element of $ S $.  We prove the following result which is a generalization of classical complex dynamics to semigroup dynamics. 
\begin{theorem}\label{cijs2}
If $ E(S) $ has non-empty interior, then $ E(S) =\mathbb{C} $.
\end{theorem}
Indeed, this theorem \ref{cijs2} is a good connection between classical transcendental dynamics and transcendental semigroup dynamics. 

The set $ E(S) $ just we defined above in the definition \ref{cir} may and may not be the set $ J(S) $. The following examples can help to compere the sets $ E(S) $ and $ J(S) $.
\begin{exm}\label{exm1}
Let $ S =\langle f, g \rangle $ be a transcendental semigroup generated by $ f(z) = \lambda \sin z $ and $ g(z) = \lambda \sin z + 2\pi $. Then $ E(S) = J(f) =J(g) $. It is also verified that $ J(S) = J(f) =J(g) $. In this case, $ E(S) = J(S) $ and so $ W(S) = F(S) $
\end{exm} 
\begin{exm}\label{exm2}
Let $ S =\langle f, g \rangle $ be a transcendental semigroup generated by $ f(z) = \lambda e^{z}, (0< \lambda < e^{-1}) $ and $ g(z) = z - e^{z}  +1 + 2\pi i$. Then $ J(f) $ is a Cantor bouquet and $ J(g) $ is different from Cantor bouquet. Both of the functions $ f $ and $ g $ have a completely invariant Fatou component. In this case, 
\end{exm}
In this paper, we prove the following result that will strengthen the examples \ref{exm1} and \ref{exm2}.
\begin{theorem}\label{nf1}
Let $ S =\langle f, g \rangle $ be a semigroup generated by two transcendental entire functions $ f $ and $ g $ such that $ J(f) \neq J(g) $. If each function $ f $ and $ g $ has a completely invariant Fatou component, then $ E(S) =\mathbb{C} $. 
\end{theorem}

\section{Results from rational semigroup dynamics}

Stankewitz \cite{stan} compered the sets $ J(S) $ and $ E(S) $ from the reference of the following two examples.
\begin{exm}{\cite[Example-1]{stan}}
Suppose that $ S = \langle f, g \rangle $ and $ J(f) = J(g) $. Then $E(S)= J(f) = J(g) $ as $ J(f) $ is completely invariant under $ f $ and $ J(g) $ is completely invariant under $ g $. Also, if $ J(f) = J(g) $, then $J(S)= J(f) = J(g)  $. So as complement of $ J(S) $ and $ E(S) $, we have $ F(S) =W(S) $. 
\end{exm}
\begin{exm}{\cite[Example-2]{stan}}
Let $ S =\langle z^{2}, z^{2}/a \rangle $, where $ a \in \mathbb{C}, |a| >1 $. Then  Julia set $ J(S)  = \{z : 1 \leq  |z| \leq  |a| \} $ which not forward invariant. So $ E(S) \neq J(S) $. In this case,  $ E(S) =\mathbb{C_{\infty}} $.  Note that $ J(f) =  \{z : |z| =1 \} $  and $ J(g) =  \{z : |z| =|a|\} $.  The Fatou set $ F(S) = \{z : |z| <1\; \text{or}\; |z| > |a| \}  $ is not S-backward invariant and so $ F(S) \neq W(S) $.  In this case, it is obvious that $ W(S) =\emptyset $. 
\end{exm}
On the basis of these two examples, Stankewitz \cite{stan} concluded that there are only two possibilities (as mentioned in examples 3.1 and 3.2)  for any polynomial semigroups. As a conclusion, he proved the following two theorems ({\cite[Theorems 1 and 2]{stan}}).
\begin{theorem}\label{sta1}
Let $ S = \langle f, g \rangle $ be a semigroup generated by two polynomials $ f $ and $ g $ of degree at least two. If $J(f) \neq J(g)  $, then $ E(S) =\mathbb{C_{\infty}} $.
\end{theorem}
 
\begin{theorem}\label{sta2}
Let $ S^{'}$ be a rational semigroup which contains two polynomials $ f $ and $ g $ of degree at least two. If $J(f) \neq J(g)  $, then $ E(S^{'}) =\mathbb{C_{\infty}} $
\end{theorem}

One of the main result in classical complex  dynamics is that if Julia set has non-empty interior, then Julia set explodes and it becomes whole complex plane. This result is generalized to completely invariant Julia set $ E(S) $. That is, Stankewitz \cite{stan} proved the following result which shows a connection between classical complex dynamics and semigroup dynamics ({\cite[Lemma 2]{stan}}).
\begin{theorem}
If $ E(S) $ has non-empty interior, then $ E(S) =\mathbb{C_{\infty}} $. 
\end{theorem}

\section{Comparison of sets $ E(S) $ and $ J(S) $ and proof of the main results}
In rational semigroup and in particular in polynomial semigroups, there are lot of studies over such completely invariant Fatou sets and Julia sets (see for instance in \cite{stan, stan2, stan1} for more detail) but there is nothing study over transcendental semigroup. So, in this paper, we concentrate more on completely invariant Julia set of transcendental semigroup $ S $. 

Next,  we workout some constructions for the comparison of sets $ E(S) $ and $ J(S) $ if semigroup $ S $ is generated by two transcendental entire functions. Let $ S =\langle f, g\rangle $ be a semigroup generated by transcendental entire functions $ f $ and $ g $. Note that $ J(h) \subset E(S) $ for all $ h \in S $ and so $ \bigcup_{h \in S}J(h) \subset E(S) $. Let us define the following countable collections of sets:
$$
\mathscr{E}_{0} = \{J(h) \}\; \text{for all} \; h \in S
$$ 
$$
\mathscr{E}_{1} = \bigcup_{h\in S}h^{-1}(\mathscr{E}_{0}) \cup \bigcup_{h\in S}h(\mathscr{E}_{0})
$$

$$
\ldots\;\;\;\;\;\;  \ldots  \;\; \;\; \;\;  \ldots   \;\;\;\;\;\;  \ldots \;\;\;\;\;\;  \ldots
$$

$$
\mathscr{E}_{n+1} = \bigcup_{h \in S}h^{-1}(\mathscr{E}_{n}) \cup \bigcup_{h \in S}h(\mathscr{E}_{n})
$$
and 
$$
\mathscr{E} = \bigcup_{n =0}^{\infty}\mathscr{E}_{n}
$$
where $ h^{-1}(\mathscr{A}) = \{h^{-1}(A): A\in \mathscr{A}\} $ and $ h(\mathscr{A}) = \{h(A): A\in \mathscr{A}\} $ for any collection of sets $ \mathscr{A} $ and a function $ h $. The following result will be the convenient description of the set $ E(S) $ of a transcendental semigroup $ S $. 
\begin{theorem}\label{cijs}
For a transcendental semigroup $ S =\langle f, g \rangle $, we have $ E(S) =\overline{\bigcup_{A\in \mathscr{E}}A} $.
\end{theorem}
\begin{proof}
By the definition of $ E(S) $, it is closed, is completely invariant under each $ h \in S $ and contains $ J(h) $ for all $ h \in S $. So we can write 
$$
E(S) \supset \overline{\bigcup_{A\in \mathscr{E}}A} 
$$
Since the set $ \overline{\bigcup_{A\in \mathscr{E}}A}  $ is closed and contains $ J(h) $ for all $ h \in S$, it remains to show that it is also completely invariant under each $ h \in S $. Since $ h $ is a continuous closed map, so under each $ h \in S $, $ h(\overline{\bigcup_{A\in \mathscr{E}}A})  $ and $ h^{-1}(\overline{\bigcup_{A\in \mathscr{E}}A})  $ are
closed sets. It proves our claim.
\end{proof}
\begin{cor}
The set $ E(S) $ is perfect.
\end{cor}
\begin{proof}
Since $ J(h) \subset E(S) $ for all $ h \in S $ and $J(h)  $ is perfect, unbounded and contains an infinite number of points for each $ h \in S $.
This corollary will be proved if we show $ E(S) $ has no isolated points. Suppose $ \alpha \in E(S) $ is an isolated point. Then it an isolated point of some $ A\in \mathscr{E} $. Choose a neighborhood $ U $ of $ \alpha $ so that $ U-\{\alpha \} \subset W(S) $ where $ W(S) $ is completely invariant Fatou set of $ S $. Since $ h^{-1}(W(S)) \subset W(S) $ and $ h(W(S)) \subset W(S) $ for all $ h\in S $, so each $ h \in S $ omits $ E(S) $ on $ U-\{\alpha \}$, which implies that every element in $ S $ is normal on $ U $. Which is a contradiction.
\end{proof}

\begin{proof}[Proof of the theorem \ref{cijs2}]
Let $ E(S)^{\circ} \neq \emptyset $, where $ E(S)^{\circ}$ denotes the interior of $ E(S) $. Then there exists a disk $ D =\{|z -z_{0}|< r\} \subset E(S)$ such that it intersects $ J(h) $ for some $ h \in S $. Then by {\cite[Theorem 3.9]{hou}}, for each finite value $ a $, there is sequence $ z_{k} \to z_{0} \in J(h) $ and a sequence of positive integers $ n_{k}\to \infty $ such that $ f^{n_{k}}(z_{k}) = a, (k =1,2,3,\ldots) $ except at most for a finite value. Then by backward invariance of $ J(h) $, $ z_{k} \in J(h) $ and by forward invariance of $ J(h) $, $ a \in J(h) $. It shows that every finite value is in $J(h)$, except at  most a single value. Since $ h \in S $ is arbitrary, so we must have $ E(S) =\mathbb{C} $.
\end{proof}
\begin{cor}
If $ E(S) \neq \mathbb{C} $, then $ W(S) $ is unbounded.
\end{cor}
\begin{proof}
If $ W(S) $ is bounded, then $ E(S) $ has interior points. By the above theorem, $ E(S) =\mathbb{C} $, which is a contradiction. 
\end{proof}
Similar to the description of the sets $ E(S) $ in lemma \ref{cijs}, we can give analogous description of the Julia set $ J(S) $ of transcendental semigroup $ S $. Let us define the following countable collections of sets:
$$
\mathscr{F}_{0} = \{J(h) \}\; \text{for all} \; h \in S
$$ 
$$
\mathscr{F}_{1} = \bigcup_{h\in S}h^{-1}(\mathscr{F}_{0}) 
$$

$$
\ldots \;\;\;\; \ldots \;\;\;\;\; \ldots 
$$

$$
\mathscr{E}_{n+1} = \bigcup_{h \in S}h^{-1}(\mathscr{F}_{n}) )
$$
and 
$$
\mathscr{F} = \bigcup_{n =0}^{\infty}\mathscr{F}_{n}
$$
where $ h^{-1}(\mathscr{A}) = \{h^{-1}(A): A\in \mathscr{A}\} $ for any collection of sets $ \mathscr{A} $ and a function $ h $. The following result will be the convenient description of the set $ J(S) $ of transcendental semigroup $ S $. 
\begin{theorem}\label{cijs1}
For a transcendental semigroup $ S =\langle f, g \rangle $, we have $ J(S) =\overline{\bigcup_{A\in \mathscr{F}}A} $.
\end{theorem}
 \begin{proof}
By the definition of $ J(S) $, it is closed, is backward invariant under each $ h \in S $ and contains $ J(h) $ for all $ h \in S $. So we can write 
$$
J(S) \supset \overline{\bigcup_{A\in \mathscr{F}}A} 
$$
Since the set $ \overline{\bigcup_{A\in \mathscr{F}}A}  $ is closed and contains $ J(h) $ for all $ h \in S$, it remains to show that it is also backward invariant under each $ h \in S $. Since $ h $ is a continuous closed map, so under each $ h \in S $, $ h^{-1}(\overline{\bigcup_{A\in \mathscr{F}}A})  $ is a
closed set. It proves our claim.
\end{proof}

\begin{cor}\label{cijs3}
For a transcendental semigroup $ S = \langle f, g \rangle $, we have $ J(S) \subset E(S) $. 
\end{cor}
\begin{proof}
Since by construction $ \mathscr{F} \subset \mathscr{E} $, so by theorems \ref{cijs}, \ref{cijs1}, assertion of the corollary follows.
\end{proof}
\begin{cor}
If $ J(S) $ has non-empty interior, then $ E(S) = \mathbb{C} $.
\end{cor}
\begin{proof}
This corollary follows from theorem \ref{cijs2} and corollary  \ref{cijs3}.
\end{proof}

Finally we prove our main result that we proposed in the theorem \ref{nf1}.
\begin{proof}[Proof of the Theorem \ref{nf1}]
Let $ U $ be a completely invariant component  of $ F(f) $. Then by {\cite[Theorem 4.36]{hou}}, $ U $ is unbounded, simply connected and $ \partial U = J(f) $. Likewise, a completely invariant component $ V $ of $ F(g) $ is unbounded, simply connected and $ \partial V = J(g) $. $ J(f) \neq J(g) $ implies that $ \partial U \neq \partial V $. By {\cite[Theorem 3.8]{hou}}, $ J(f) $ and $ J(g) $ are unbounded, so $ U \cap J(g) \neq \emptyset $ and $ V \cap J(f) \neq \emptyset $. The fact $ \partial U \neq \partial V $ implies that $ J(f) $ must intersect interior of $ V $ and $ J(g) $ must intersect interior of $ U $. 

Let $ z \in J(f) \cap V^{\circ} $, where $ V^{\circ} $ is an interior of $ V $. Then by the forward invariance   of $ E(S) $ and $ V^{\circ} $ under the map $ g $, so, its nth iterates, that is, $ g^{n}(z) \in E(S)$ and $ g^{n}(z) \in V^{\circ}$ for all $ n \in \mathbb{N} $. Likewise, we can write $ f^{n}(z) \in E(S)$ and $ f^{n}(z) \in U^{\circ}$ for all $ n \in \mathbb{N} $. This shows that $ E(S) $ intersects open sets $ U^{\circ} $ and $ V^{\circ} $. So, $ E(S) $ intersects $U^{\circ} \cap V^{\circ}  $. Since $ E(S) $ is perfect and completely invariant set, so it contains all limits of the sequences $ (f^{n}) $ and $ (g^{n}) $. This prove that $ E^{\circ} \neq \emptyset$ and hence $ E(S) =\mathbb{C} $. 
\end{proof}


\begin{thebibliography}{30}



\bibitem {hin} Hinkkanen, A. and Martin, G.J.: \textit{The dynamics of semigroups of rational functions- I}, Proc. London Math. Soc. (3) 73, 358-384, (1996).

\bibitem {hou} Hua, X.H. and Yang, C.C.: \textit {Dynamic of transcendental functions}, Gordon and Breach Science Publication, (1998).


\bibitem {poo} Poon, K.K.: \textit{Fatou-Julia theory on transcendental semigroups}, Bull. Austral. Math. Soc. Vol- 58(1998) PP 403-410.



\bibitem {stan} Stankewitz, R.: \textit{Completely invariant Julia set of polynomial semigroup}, Proc. Amer. Math. Soc. Vil. 127, No. 10 (1999), 2889-2898.

\bibitem {stan2} Stankewitz, R.: \textit{Completely invariant sets of normality for rational semigroups},Complex Variables, Theory and Applications, Int. Journal, Vol. 40 (2007), 199-210.

\bibitem {stan1} Stankewitz, R., Sugawa, T. and Sumi, H.: \textit{Some counter examples in dynamics of rational semigroups}, arXiv:0708.3434v1 [math DS], 2007.


\end{thebibliography}
\end{document}